\documentclass[12pt]{article}

\usepackage{t1enc}
\usepackage[latin1]{inputenc}
\usepackage[english]{babel}
\usepackage{amsmath,amsthm}
\usepackage{amsfonts}
\usepackage{latexsym}
\usepackage{graphicx}
\usepackage{txfonts}

\newtheorem{thm}{Theorem}
\newtheorem{lem}[thm]{Lemma}

\newtheorem{conj}[thm]{Conjecture}

\newtheorem{cm}{Claim}

\date{}

\begin{document}

\title{List total coloring of pseudo-outerplanar graphs\thanks{This research is partially supported by the National Natural Science Foundation of China (No.\,11101243, 11201440) and the Fundamental Research Funds for the Central Universities.}}
\author{Xin Zhang\thanks{Email address: xzhang@xidian.edu.cn.}\\
{\small Department of Mathematics, Xidian University, Xi'an 710071, China}}
\date{}

\maketitle

\begin{abstract}\baselineskip  0.6cm
A graph is pseudo-outerplanar if each of its blocks has an embedding in the plane so that the vertices lie on a fixed circle and the edges lie inside the disk of this circle with each of them crossing at most one another. It is proved that every pseudo-outerplanar graph with maximum degree $\Delta\geq 5$ is totally $(\Delta+1)$-choosable.\\[.2em]
Keywords: pseudo-outerplanar graph; list total coloring; structures.
\end{abstract}

\baselineskip  0.6cm

\section{Introduction}

In this paper, all graphs are finite, simple and undirected. By $V(G)$, $E(G)$,
$\delta(G)$ and $\Delta (G)$, we denote the vertex set, the edge set, the minimum degree and the maximum degree
of a graph $G$, respectively. By $VE(G)$, we denote the set $V(G)\cup E(G)$. For undefined concepts we refer the reader to
\cite{Bondy}.

A total coloring of a graph $G$ is an assignment of colors to the vertices and edges of $G$ such that every pair of adjacent/incident elements receive different colors. A \emph{$k$-total coloring} of a graph $G$ is a total coloring of $G$ from a set of $k$ colors. The minimum positive integer $k$ for which $G$ has a $k$-total coloring, denoted by $\chi''(G)$, is the \emph{total chromatic number} of $G$.

Suppose that a set $L(x)$ of colors, called a \emph{list} of $x$, is assigned to each element $x\in VE(G)$. A total coloring $\varphi$ is called a \emph{list total coloring} of $G$ for $L$, or an \emph{$L$-total coloring}, if $\varphi(x)\in VE(G)$ for each element $x\in VE(G)$. If $|L(x)|\equiv k$ for every
$x\in VE(G)$, then an $L$-total coloring is called a \emph{list $k$-total coloring} and we say that $G$ is \emph{$k$-totally choosable}.
The minimum integer $k$ for which $G$ has a list $k$-total coloring, denoted by $\chi''_l(G)$, is the \emph{list total chromatic number} of $G$. It is obvious that $\chi''_l(G)\geq \chi''(G)\geq \Delta(G)+1$.

In 1997, Borodin, Kostochka and Woodall \cite{Borodin.lcc} raised the following conjecture, which is known as list total conjecture (LTC). In the same paper, they gives an affirmative answer to LTC for planar graphs with maximum degree at least 12.

\begin{conj}\label{conj}
For any graph $G$, $\chi''_{l}(G)=\chi''(G)$.
\end{conj}

Recently, LTC was investigated by many authors including \cite{Hou1,Hou2,Juvan,Liu1,Liu2,Liu3,Wang2,WJL,Zhou}. In particular, Wang and Lih \cite{Wang2} confirmed LTC for outerplanar graphs with maximum degree at least 4, and this result was generalized to series-parallel graphs by Zhou, Matsuo and Nishizeki \cite{Zhou} in 2005. However, this ``list total conjecture'' is still very much open.

In this paper, we investigate the list total colorings of pseudo-outerplanar graphs, another class of graphs (different from series-parallel graphs) between outerplanar graphs and planar graphs. A graph is \emph{pseudo-outerplanar} if each of its blocks has an embedding in the plane so that the vertices lie on a fixed circle and the edges lie inside the disk of this circle with each of them crossing at most one another. For example, $K_{2,3}$ and $K_4$ are both pseudo-outerplanar graphs. The concept of pseudo-outerplanar graph was first introduced by Zhang, Liu and Wu \cite{ZPOPG} in 2012. They proved that the class of outerplanar graphs is the intersection of the classes of pseudo-outerplanar graphs and series-parallel graphs.

The purpose of this paper is to show that LTC holds for pseudo-outerplanar graphs with maximum degree at least 5, and thus extending one result of \cite{Arxiv}, where is proved that every pseudo-outerplanar graphs with maximum degree at least 5 is totally $(\Delta+1)$-colorable.

\section{Structural properties of PO-graphs}

In what follows, we always assume that every pseudo-outerplanar graph $H$ considered in this paper has been drawn on the plane so that its pseudo-outerplanarity is satisfied; and call such a drawing a \emph{pseudo-outerplanar diagram}.
Let $H$ be a pseudo-outerplanar diagram and let $G$ be a block of $G$. Denote by $v_1, v_2, \ldots, v_{|G|}$ the vertices of $G$ that lie in a clockwise sequence. Let $\mathcal{V}[v_i,v_j]=\{v_i, v_{i+1}, \ldots, v_j\}$ and $\mathcal{V}(v_i,v_j)=\mathcal{V}[v_i,v_j]\backslash \{v_i,v_j\}$, where the subscripts are taken modular $|B|$. A vertex set $\mathcal{V}[v_i,v_j]$ is a \emph{non-edge} if $j=i+1$ and $v_iv_j\not\in E(G)$, is a \emph{path} if $v_k v_{k+1}\in E(G)$ for all $i\leq k<j$, and is a \emph{subpath} if $j>i+1$ and some edges in the form $v_kv_{k+1}$ for $i\leq k<j$ are missing. An edge $v_iv_j$ in $G$ is a \emph{chord} if $j-i\neq 1$ or $1-|G|$. By $\mathcal{C}[v_i,v_j]$, we denote the set of chords $xy$ with $x,y\in \mathcal{V}[v_i,v_j]$. We say that a chord $v_kv_l$ is contained in a chord $v_iv_j$ if $i\leq k\leq l\leq j$.

\begin{lem}\label{path}
Let $v_i$ and $v_j$ be vertices of a 2-connected pseudo-outerplanar diagram $G$. If there is no crossed chords in $\mathcal{C}[v_i,v_j]$ and no edges between
$\mathcal{V}(v_i,v_j)$ and $\mathcal{V}(v_{j},v_{i})$, then $\mathcal{V}[v_i,v_j]$ is either non-edge or path.
\end{lem}

\begin{proof}
The proof is same to the one of Claim 1 in \cite{ZPOPG}, we refer the readers to [12, pp.2794].
\end{proof}

\begin{lem}\label{opg}{\rm \cite{Wang}}
Each outerplanar graph $G$ with minimum degree at least 2 contains a $2$-vertex that is adjacent to a $4^-$-vertex.
\end{lem}

\begin{thm}\label{str}
Each pseudo-outerplanar graph $G$ with minimum degree at least 2 contains at least one of the following configurations:\\
$(a)$ a $2$-vertex $u$ adjacent to a $4^-$-vertex $v$;\\
$(b)$ a path $v_1u_1v_2u_2v_3u_3v_4$ so that $v_1v_2,v_1v_3,v_2v_3,v_2v_4,v_3v_4\in E(G)$, $d(u_1)=d(u_2)=d(u_3)=2$ and $d(v_2)=d(v_3)=5$;\\
$(c)$ a cycle $u_1u_2u_3u_4$ so that $d(u_2)=d(u_4)=2$;\\
$(d)$ a cycle $u_1u_2u_3u_4$ so that $u_2u_4\in E(G)$, $d(u_2)=d(u_4)=3$ and $d(u_3)\leq 4$;\\
$(e)$ a cycle $u_1u_2u_3u_4$ so that $u_2u_4\in E(G)$, $d(u_2)=d(u_4)=3$ and $u_3$ is adjacent to a $2$-vertex $v$;\\
$(f)$ a cycle $u_1u_2u_3u_4$ so that $u_2u_4\in E(G)$, $d(u_2)=d(u_4)=3$ and $u_3$ is adjacent to a $3$-vertex $v$ and a vertex $x$ with $vx\in E(G)$;\\
$(g)$ a cycle $u_1u_2u_3u_4$ so that $u_1u_3, u_2u_4\in E(G)$, $d(u_2)=d(u_4)=3$ and $u_3$ is adjacent to a vertex $v$ with $u_1v\in E(G)$.
\end{thm}

\begin{proof}
We first assume that $G$ is a 2-connected pseudo-outerplanar diagram with $v_1\ldots v_{|G|}$ being the vertices of this diagram that lie in a clockwise sequence. If $G$ contains no crossings, then $G$ is outerplanar, which implies that $G$ contains (a) by Lemma \ref{opg}. If $G$ contains a crossing, then we can choose one pair of crossed chords $v_iv_j$ and $v_kv_j$ such that\\[.5em]
(1) $v_iv_j$ crosses $v_kv_l$ in $G$;\\
(2) $v_i,v_k,v_j$ and $v_l$ lie in a clockwise sequence;\\
(3) besides $v_iv_j$ and $v_kv_l$, there is no crossed chord in $\mathcal{C}[v_i,v_l]$.\\[.5em]
Suppose that this theorem is false. By a same proof of Theorem 4.2 in \cite{ZPOPG}, here we can prove that
\begin{equation}\label{eq1}
    l-j=j-k=k-i=1~{\rm and}~v_iv_k,v_kv_j,v_jv_l\in E(G),
\end{equation}
since $G$ does not contain (a), (b) or (c). This pair of crossed chords $v_iv_j$ and $v_kv_j$ satisfying \eqref{eq1} are called \emph{co-crossed chords}.

Since the configuration (d) is absent from $G$, $\min\{d(v_i),d(v_l)\}\geq 5$. This implies that there is at least one chord $v_lv_s$ so that $s\neq i,k$ and at least one chord $v_mv_i$ so that $m\neq j,l$. We now choose $s$ and $m$ so that there is no chord $v_lv_t$ contained in $v_lv_s$ and no chord $v_iv_n$ contained in $v_iv_m$. In the following, we call the graph induced by $v_iv_j,v_kv_l,v_iv_k,v_kv_j,v_jv_l,v_lv_s$ or by $v_iv_j,v_kv_l,v_iv_k,v_kv_j,v_jv_l,v_iv_m$ an \emph{inner cluster} of $G$, denoted by $IC(i,l,s)$ or $IC(m,i,l)$, respectively. The \emph{width} of the two inner clusters defined above is $|\mathcal{V}[v_i,v_s]|$ and $|\mathcal{V}[v_m,v_l]|$, respectively.

\begin{cm}\label{cm1}
If $IC(i,l,s)$ is an inner clusters with the shortest
width among all the
inner clusters that contained in the graphs induced by $\mathcal{V}[v_i,v_s]$, then the chords $v_lv_s$ is crossed.
\end{cm}

\proof Without loss of generality, assume that $i=1$ and $l=4$.
If $v_4v_s$ is a non-crossed chord, then there is no edges between $\mathcal{V}(v_4,v_s)$ and $\mathcal{V}(v_s,v_4)$.
If there is no chords contained in $v_4v_s$, then (a) or (e) would appear in $G$. If there are chords contained in $v_4v_s$, then we consider two cases.

\emph{Case $1.1$. Every chord contained in $v_4v_s$ is non-crossed.}

If every chord contained in $v_4v_s$ is non-crossed, then by Lemma \ref{path}, $\mathcal{V}[v_4,v_s]$ is a path. We now claim that
there exists a chord in $S:=\mathcal{C}[v_4,v_s]\setminus \{v_4v_s\}$ that contains at least one other chord. If this proposition does not hold, then we can choose one chord $v_iv_j$ with $4<i<j\leq s$ so that $v_iv_j$ contains no other chords. If $|j-i|\geq 3$, then we can find two adjacent 2-vertices in $\mathcal{V}[v_i,v_j]$, a contradiction. If $|j-i|=2$, then $d(v_{i+1})=2$ and $d(v_i)\geq 5$. This implies that besides $v_iv_j$, there are at least two non-crossed chords that are incident with $v_i$, therefore, we would find two chords in $S$ so that one contains the other in $G$, a contradiction. Hence, we can choose a chord $v_iv_j$ in $S$ so that $v_iv_j$ contains at least one other chord, say $v_av_b$, and moreover, every chord contained in $v_iv_j$ contains no other chords. Without loss of generality, assume that $b\neq j$.
If $|b-a|\geq 3$, then we can find two adjacent 2-vertices in $\mathcal{V}[v_a,v_b]$, a contradiction. If $|b-a|=2$, then $d(v_{a+1})=2$ and $d(v_b)\geq 5$. This implies that besides $v_av_b$, there are at least two non-crossed chords that are incident with $v_b$, therefore, we would find two chords in $\mathcal[v_i,v_j]\setminus \{v_iv_j\}$ so that one contains the other, a contradiction to our assumption.

\emph{Case $1.2$. There is at least one pair of crossed chords that are contained in $\mathcal{C}[v_4,v_s]$.}

If there is at least one pair of crossed chord that are contained in $v_4v_s$, then we choose one pair of co-crossed chords $v_av_b$ and $v_cv_d$ with $c-a=b-c=d-b=1$ and $v_av_c,v_cv_b,v_bv_d\in E(G)$. Since the configuration (d) is absent from $G$, $\min\{d(v_a),d(v_d)\}\geq 5$.
If $a=4$, then it is easy to see that (f) occurs in $G$. If $d=s$, then there exists an inner cluster $IC(x,a,d)$ with $4\leq x<a$ and width $|\mathcal{V}[v_x,v_d]|<s$, a contradiction. Therefore, we assume that
$a\neq 4$ and $d\neq s$. Since $IC(1,4,s)$ is an inner cluster with the shortest width in $G$, there is no chord in the form $v_av_i$ with $4\leq i<a$ or in the form $v_bv_j$ with $d<j\leq s$. Since the configuration (d) is absent from $G$, $\min\{d(v_a),d(v_d)\}\geq 5$. The above two facts imply that $v_av_d\in E(G)$ and there are a chord $v_iv_d$ with $4\leq i<a$ and a chord $v_av_j$ with $d< j\leq s$. We call the graph induced by $v_av_b,v_av_c,v_av_d,v_bv_c,v_bv_d.v_cv_d,v_av_i$ and $v_dv_j$ a \emph{$K_4$-cluster} derived from $v_av_b$ and $v_cv_d$, and by $|\mathcal{V}[v_i,v_j]|$, we denote the \emph{width} of this $K_4$-cluster. Without loss of generality, we assume that the width of the above $K_4$-cluster is the shortest among all the $K_4$-clusters contained in the graph induced by $\mathcal{V}[v_4,v_s]$.
If there is no crossed chords in $\mathcal{C}[v_d,v_j]$, then by Lemma \ref{path}, $\mathcal{V}[v_d,v_j]$ is either a non-edge or a path, because there is no edges between $\mathcal{V}(v_d,v_j)$ and $\mathcal{V}(v_j,v_d)$. Since $d(v_d)\geq 5$, $\mathcal{V}[v_d,v_j]$ cannot be a non-edge, thus it is a path.
If there are no chords that are contained in $\mathcal{V}[v_d,v_j]$, then either (a), (e) or (g) would occur in $G$.
If there are chords contained in $\mathcal{V}[v_d,v_j]$, then
by similar arguments as in Case 1.1, one can prove that there is no non-crossed chords contained in $\mathcal{V}[v_d,v_j]$.
If there is at least one pair of co-crossed chords $v_{a'}v_{b'}$ and $v_{c'}v_{d'}$ with $a'<c'<b'<d'$ that are contained in $\mathcal{V}[v_d,v_j]$, then $a'\neq d$, because otherwise $IC(a,d,b')$ would be an inner cluster shorter than $IC(1,4,s)$, a contradiction. This implies,
by similar arguments as above, that either there is an inner cluster $IC(x,a',d')$ with $d\leq x<a'$ and width $|\mathcal{V}[v_x,v_{d'}]|<s$,
or $d'\neq j$ and there is an inner cluster $IC(a',d',y)$ with $d'<y\leq j$ and width $|\mathcal{V}[v_{a'},v_{y}]|<s$, or
$d'\neq j$ and
there is a $K_4$-cluster derived from $v_{a'}v_{b'}$ and $v_{c'}v_{d'}$ with width no more than $|\mathcal{V}[v_d,v_j]|<|\mathcal{V}[v_i,v_j]|$. In either case, we would obtain a contradiction to our assumption.

Hence, the chord $v_4v_s$ is crossed. \hfill$\|$

\begin{cm}\label{cm2}
If $IC(i,l,s)$ is an inner clusters with the shortest
width among all the
inner clusters that contained in the graphs induced by $\mathcal{V}[v_i,v_s]$, then the chords $v_lv_s$ cannot be crossed.
\end{cm}

\proof Without loss of generality, assume that $i=1$ and $l=4$.
Suppose, to the contrary, that $v_lv_s$ is crossed by one other chord $v_av_b$ with $4<a<s$.
 If there is at least one pair of crossed chords that are contained in $\mathcal{C}[v_4,v_a]$ or $\mathcal{C}[v_a,v_s]$, then by similar arguments as in Case 1.2, one can obtain contradictions. Therefore, every chord contained in $\mathcal{C}[v_4,v_a]$ or $\mathcal{C}[v_a,v_s]$ is non-crossed. Since there is no edges between $\mathcal{V}(v_4,v_a)$ and $\mathcal{V}(v_a,v_4)$, or between $\mathcal{V}(v_a,v_s)$ and $\mathcal{V}(v_s,v_a)$, by Lemma \ref{path}, $\mathcal{V}[v_4,v_a]$ or $\mathcal{V}[v_a,v_s]$ is either non-edge or path.
If $\mathcal{V}[v_4,v_a]$ and $\mathcal{V}[v_a,v_s]$ are non-edges, then $d(v_a)=1$, a contradiction.
If $\mathcal{V}[v_4,v_a]$ and $\mathcal{V}[v_a,v_s]$ are paths, then by similar arguments as in Case 1.1, (a) or (f) would appear in $G$.
If $\mathcal{V}[v_4,v_a]$ is path and $\mathcal{V}[v_a,v_s]$ is non-edge, then by similar arguments as in Case 1.1, (a) would appear in $G$ unless $a=5$, in which case (e) occurs in $G$.
Hence, we assume that $\mathcal{V}[v_4,v_a]$ is non-edge and $\mathcal{C}[v_a,v_s]$ is path in the following. By similar arguments as in Case 1.1, one can obtain contradictions if $s-a\geq 2$, so assume that $s-a=1$, that is, $a=5$ and $s=6$.
Since $d(v_a)=2$, $b\neq 1$, because otherwise we would find (e).
In the following, the graph induced by $v_1v_2,v_2v_3,v_3v_4,v_1v_3,v_2v_4.v_4v_6,v_5v_6$ and $v_5v_b$ (or a graph isomorphic to this graph) is called a \emph{$\rtimes$-cluster}, and the \emph{width} of this
$\rtimes$-cluster is $|\mathcal{V}[v_1,v_b]|$. Without loss of generality, we can assume the width of the above $\rtimes$-cluster is the shortest among all the $\rtimes$-clusters that are contained in the graph induced by $\mathcal{V}[v_1,v_b]$.

If there is no chords contained in $v_6v_b$, then (a) appears in $G$, so we assume that there are chords contained in $v_6v_b$.
If every chord contained in $v_6v_b$ is non-crossed, then by Lemma \ref{path}, $\mathcal{V}[v_6,v_b]$ is either non-edge or path.
If $\mathcal{V}[v_6,v_b]$ is a non-edge, then $v_5$ and $v_6$ are two adjacent 2-vertices, a contradiction, so we assume that $\mathcal{V}[v_6,v_b]$ is a path. In this case, we can can use similar arguments as in Case 1.1 to obtain contradictions. Therefore, we shall assume that there is at least one pair of crossed chords that are contained in $\mathcal{C}[v_6,v_b]$.

We arbitrarily choose one pair of co-crossed chords $v_{i'}v_{j'}$ and $v_{k'}v_{l'}$ with $i'<k'<j'<l'$ that are contained in $\mathcal{C}[v_6,v_b]$. Since both $v_6$ and $v_b$ are adjacent to a 2-vertex $v_5$, $i'\neq 6$ and $l'\neq b$, because otherwise we would find (e) in $G$. Due to the absence of (d), we have $\min\{d(v_{i'}),d(v_{l'})\}\geq 5$, which implies that there exists $s'\neq i',k'$ and $m'\neq j',l'$ so that $v_{l'}v_{s'}$ and $v_{i'}v_{m'}$ are chords in $G$. If $l'<s'\leq b$, then we can assume, without loss of generality, that $IC(i',l',s')$ is an inner cluster with the shortest width among all the inner clusters contained in the graph induced by $\mathcal{V}[v_{i'},v_{s'}]$. If $v_{l'}v_{s'}$ is a non-crossed chord, then we use similar arguments as in the proof of Claim \ref{cm1} to obtain contradictions. If $v_{l'}v_{s'}$ is a chord crossed by one other chord $v_{a'}v_{b'}$ with $l'<a'<s'$, then by similar arguments as in the first part of this proof, one can deduce that $s'-a'=a'-l'=1$, $v_{l'}v_{a'}\not\in E(G)$ and $v_{a'}v_{s'}\in E(G)$. This implies that $6\leq b'<i'$, because otherwise we would find a shorter $\rtimes$-cluster, a contradiction to our assumption. Since $v_{a'}v_{b'}$ has already crossed $v_{l'}v_{s'}$ in $G$, $b'\leq m'<i'$. Without loss of generality, assume that $IC(m',i',l')$ is an inner cluster with the shortest width among all the inner clusters contained in the graph induced by $\mathcal{V}[v_{m'},v_{l'}]$.
If $v_{i'}v_{m'}$ is a non-crossed chord, then we use similar arguments as in the proof of Claim \ref{cm1} to obtain contradictions. If $v_{i'}v_{m'}$ is crossed by one other chord $v_{c'}v_{d'}$ with $m'<c'<i'$, then by similar arguments as in the first part of this proof, one can deduce that $m'-c'=c'-i'=1$, $v_{c'}v_{i'}\not\in E(G)$ and $v_{c'}v_{m'}\in E(G)$. Since $v_{a'}v_{b'}$ has already crossed $v_{l'}v_{s'}$ in $G$, $b'\leq d'<m'$, which implies a shorter $\rtimes$-cluster that is contained in the graph induced by $\mathcal{V}[v_1,v_b]$, a contradiction to our assumption. Therefore, for any chord $v_{l'}v_{s'}$ with $s'\neq i',k'$, we have $6\leq s'<i'$. Similarly, we can prove, for any chord $v_{i'}v_{m'}$ with $m'\neq j',l'$, that $l'<m'\leq b$. Since $\min\{d(v_{i'}),d(v_{l'})\}\geq 5$, $v_{i'}v_{l'}\in E(G)$, which implies a $K_4$-cluster derived from $v_{i'}v_{j'}$ and $v_{k'}v_{l'}$. Without loss of generality, assume that the width of this $K_4$-cluster is the shortest among all the $K_4$-clusters contained in the graph induced by $\mathcal{V}[v_{m'},v_{s'}]$, then by similarly arguments as in the proof of Claim \ref{cm1}, we can obtain contradictions. \hfill$\|$

\vspace{2mm} It is easy to see that the above two claims are conflicting. Hence, every 2-connected pseudo-outerplanar graphs contains one of the configurations among (a)--(g). We now assume that $G$ has cut vertices and choose one of its end-blocks. Denote the chosen end-block by $B$ and the vertices of this end-block that lie in a clockwise sequence by $v_1,\ldots,v_{|B|}$. Without loss of generality, assume that $v_1$ is the unique cut-vertex of $B$.

First, assume that there are no crossings in the end-block $B$. Since $B$ is a 2-connected outerplanar graph, $B$ is hamiltonian, which implies that $\mathcal{V}[v_1,v_{|B|}]$ is a path. If there is at most one chord in $B$, then it is easy to see that $G$ contains (a). If there are two chords $v_iv_j$ and $v_sv_t$ in $B$, then without loss of generality, we can assume that $1\leq j<t<s<i$ ,
therefore, by similar arguments as in Subcase 1.1, one can prove that $G$ contains (a).

At last, assume that there is at least one pair of crossed chords $v_iv_j$ and $v_kv_l$ in $B$. Without loss of generality, assume that $1<i<k<j<l\leq |B|$ and that $v_iv_j$ and $v_kv_l$ are a pair of co-crossed chords (so they satisfy \eqref{eq1}).
Since (d) is absent from $G$, $\min\{d_B(v_i),d_B(v_l)\}\geq 5$.

If there is a vertex $v_s$ with $l<s\leq |B|$ or $s=1$ so that $v_lv_s$ is a non-crossed chord, then by the proof of Claim \ref{cm1}, one can find one of the configurations (a)--(g) in the graph induced by $\mathcal{V}[v_i,v_s]$, and moreover, $v_1$ is not the vertices with bounded degree in the configuration. If there is a vertex $v_m$ with $1\leq m<i$ so that $v_iv_m$ is a non-crossed chord, then we can prove the theorem similarly. Therefore,
we have
\begin{cm}\label{cm3}
There do not exist vertex $v_s$ with $l<s\leq |B|$ or $s=1$ so that $v_lv_s$ is a non-crossed chord or vertex $v_m$ with $1\leq m<i$ so that $v_iv_m$ is a non-crossed chord.\hfill$\|$
\end{cm}

Suppose that there is a vertex $v_s$ with $l<s\leq |B|$ or $s=1$ so that $v_lv_s$ is a chord crossed by one other chord $v_av_b$ with $l<a<s$.
If the graph induced by $v_iv_k,v_kv_j,v_jv_l,v_iv_j,v_kv_l,v_lv_s,v_sv_a$ and $v_av_b$ is not a $\rtimes$-cluster, then by the proof of Claim \ref{cm2}, one can find one of the configurations (a)--(g) in the graph induced by $\mathcal{V}[v_i,v_s]$, and thus in $G$.
If $s<b\leq |B|$ or $b=1$, then
by the proof of Claim \ref{cm2}, one can also find one of the configurations (a)--(g) in the graph induced by $\mathcal{V}[v_i,v_b]$, and moreover, $v_1$ is not the vertices with bounded degree in the configuration. Thus, we have $a-l=s-a=1$, $v_lv_a\not\in E(G)$, $v_av_s\in E(G)$ and $1<b\leq i$.
If $b=i$, then it is easy to prove that $d_B(v_l)\leq 4$, a contradiction.
If $b\neq i$, then there is a chord $v_mv_i$ with $b\leq m<i$, since
$d_B(v_i)\geq 5$ and $v_av_b$ is crossed by $v_lv_s$.
By Claim \ref{cm3}, $v_mv_i$ is a crossed chord, and we assume that it is crossed by $v_nv_t$ with $m<n<i$. Similarly as above, we shall also assume that $i-n=n-m=1$, $v_nv_i\not\in E(G)$ and $v_mv_n\in E(G)$.
If $b\leq t<m$, then by similar arguments as in the proof of Claim \ref{cm2}, one can find one of the configurations (a)--(g) in the graph induced by $\mathcal{V}[v_t,v_l]$, and thus in $G$. If $t=l$, then $d_B(v_i)\leq 4$, a contradiction. Therefore, we immediately deduce the following claim.
\begin{cm}\label{cm4}
There do not exist vertex $v_s$ with $l<s\leq |B|$ or $s=1$ so that $v_lv_s$ is a crossed chord, and similarly, there do not exist $v_m$ with $1\leq m<i$ so that $v_mv_i$ is a crossed chord.\hfill$\|$
\end{cm}

Since $\min\{d_B(v_i),d_B(v_l)\}\geq 5$, by Claims \ref{cm3} and \ref{cm4},
there exist $v_s$ with $1<s<i$ and $v_m$ with $l<m\leq |B|$ so that $v_lv_s$ and $v_iv_m$ are two chords that cross each other.
If there is at least one pair of crossed chords that are contained in $\mathcal{C}[v_l,v_m]$, then by similar arguments as in Case 1.2, one can obtain contradictions. If every chord contained in $\mathcal{C}[v_l,v_m]$ is non-crossed, then by Lemma \ref{path}, $\mathcal{V}[v_l,v_m]$ is either a path or a non-edge. However, if $\mathcal{V}[v_l,v_m]$ is a path with $m-l\geq 2$, then by similar arguments as in Case 1.1, one can find (a) or (e) in $G$; if
$\mathcal{V}[v_l,v_m]$ is a path with $m-l=1$, then (g) occurs in $G$, since $d_B(v_l)\geq 5$ implies $v_iv_l\in E(G)$; and if $\mathcal{V}[v_l,v_m]$ is a non-edge, then $d_B(v_l)\leq 4$, a contradiction.
\end{proof}

\section{List total coloring of PO-graphs}

In this section we present a sufficient condition for a pseudo-outerplanar graph to have
a list total coloring and prove the following theorem.

\begin{thm}\label{listtotal}
Let $G$ be a pseudo-outerplanar graph, and let $L$ be a list of $G$. If $$|L(x)|\geq \max\{6,\Delta(G)+1\}$$ for each $x\in VE(G)$, then $G$ has an $L$-total coloring.
\end{thm}

Before proving Theorem \ref{listtotal}, we introduce some necessary notations. Let $L$ be a list of a graph $G$ and let $L'$ be a list of a graph $G'\subset G$ with $L'(x)=L(x)$ for each element $x\in VE(G)$. Suppose that we have already obtained an $L'$-total coloring $\varphi'$ of $G'$, and that we are to extend $\varphi'$ to an $L$-total coloring $\varphi$ of $G$ without altering the colors in $G'$.
For each $x\in VE(G)$, let $L_{av}(x,\varphi')$ be the \emph{available list}
(the set of all colors in $L(x)$ that are available) for $x$ when $\varphi'$ is extended to an $L$-total coloring $\varphi$ of $G$.

\begin{lem}\label{b}
Suppose that $G$ contains a path $P=v_1u_1v_2u_2v_3u_3v_4$ so that $v_1v_2,v_1v_3,v_3v_4\in E(G)$ and $d(u_1)=d(u_2)=d(u_3)=2$.
Let $\varphi'$ be a partial $L$-total coloring of $G$ so that the uncolored elements under $\varphi'$ are $u_1,u_2,u_3,v_2,v_3,v_1v_2,v_2v_3,v_3v_4$ and the edges of the path $P$, where $L$ is a list assignment of $G$. If
$$\min\{|L_{av}(u_1v_1,\varphi')|,|L_{av}(u_3v_4,\varphi')|,|L_{av}(v_1v_2,\varphi')|,|L_{av}(v_3v_4,\varphi')|\}\geq 2,$$
$$\min\{|L_{av}(v_2,\varphi')|,|L_{av}(v_3,\varphi')|\}\geq 3,$$
$$|L_{av}(v_2v_3,\varphi')|\geq 4,$$and
$$\min\{|L_{av}(u_1v_2,\varphi')|,|L_{av}(u_2v_2,\varphi')|,|L_{av}(u_2v_3,\varphi')|,|L_{av}(u_3v_3,\varphi')|\}\geq 5,$$
then $\varphi'$ can be extended to an $L$-total coloring $\varphi$ of $G$ without altering the colors in $G'$.
\end{lem}

\begin{proof}
Without loss of generality, we assume that
$|L_{av}(u_1v_1,\varphi')|=|L_{av}(u_3v_4,\varphi')|=|L_{av}(v_1v_2,\varphi')|=|L_{av}(v_3v_4,\varphi')|=2,$
$|L_{av}(v_2,\varphi')|=|L_{av}(v_3,\varphi')|=3,$
$|L_{av}(v_2v_3,\varphi')|=4,$ and
$|L_{av}(u_1v_2,\varphi')|=|L_{av}(u_2v_2,\varphi')|=|L_{av}(u_2v_3,\varphi')|=|L_{av}(u_3v_3,\varphi')|=5.$ (otherwise we can shorten some lists that assigned to the elements of $VE(G)$ so that those conditions are satisfied).
We extend $\varphi'$ to an $L$-total coloring $\varphi$ of $G$ by two stages.

\vspace{2mm}\noindent \emph{\textbf{Stage 1}. Color $u_3v_3,u_3v_4,v_3v_4$ and $v_3$ so that the resulted partial coloring $\varphi^1$ satisfies one of the following conditions:\\
$(1)$ $|L_{av}(v_2v_3,\varphi^1)|\geq 3$,\\
$(2)$ $|L_{av}(v_2v_3,\varphi^1)|=2$~and~$L_{av}(v_2v_3,\varphi^1)\neq L_{av}(v_2,\varphi^1)$~if~$|L_{av}(v_2,\varphi^1)|=2$,\\
$(3)$ $|L_{av}(v_2v_3,\varphi^1)|=2$~and~$L_{av}(v_1v_2,\varphi^1)\neq L_{av}(v_2,\varphi^1)$~if~$|L_{av}(v_2,\varphi^1)|=2$.
}

\vspace{2mm}We now prove that the coloring $\varphi^1$ constructed in stage 1 exists.
Assume that $L_{av}(v_3v_4,\varphi')=\{1,2\}$. We now color $v_3$ with a color, say 3, from $L_{av}(v_3,\varphi')\setminus \{1,2\}$.

\noindent\emph{\underline{Case 1}. $2\in L_{av}(u_3v_4,\varphi')$ (the case when $1\in L_{av}(u_3v_4,\varphi')$ is similar).}

Color $v_3v_4$ and $u_3v_4$ with 1 and 2, and then discuss two subcases.

\noindent\emph{\underline{Case 1.1}. $\{1,2,3\}\subseteq L_{av}(u_3v_3,\varphi')$ }

Assume that $L_{av}(u_3v_3,\varphi')=\{1,2,3,4,5\}$. Denote the current partial coloring by $\phi_0$.

If $4\not\in L_{av}(v_2v_3,\varphi')\setminus \{1,3\}$, then color $u_3v_3$ with 4, and we have $\{1,3\}\subseteq L_{av}(v_2v_3,\varphi')$, otherwise $\phi_0$ satisfies (1) and we let $\varphi^1:=\phi_0$.
Let $L_{av}(v_2v_3,\varphi')=\{1,3,n_1,n_2\}$, where $\{1,3\}\cap \{n_1,n_2\}=\emptyset$. If $L_{av}(v_2,\varphi')\neq \{3,n_1,n_2\}$, then $\phi_0$ satisfies (2) and we let $\varphi^1:=\phi_0$, so we assume that $L_{av}(v_2,\varphi')= \{3,n_1,n_2\}$. Now we erase the color on $u_3v_3$ and recolor $v_3v_4$ and $u_3v_4$ with 2 and a color $\phi_1(u_3v_4)\in L_{av}(u_3v_4,\varphi')\setminus \{2\}$, respectively.
If $\phi_1(u_3v_4)\neq 4$, then color $u_3v_3$ with 4. Since the current coloring $\phi_1$ satisfies (1) or (2), we let $\varphi^1:=\phi_1$.
If $\phi_1(u_3v_4)= 4$, then $L_{av}(u_3v_4,\varphi')=\{2,4\}$, and color $u_3v_3$ with 5. If $\{n_1,n_2\}\neq \{2,5\}$, then the current coloring $\phi_1$ satisfies (1) or (2), so let $\varphi^1:=\phi_1$.
If $\{n_1,n_2\}=\{2,5\}$, then recolor $u_3v_3$ with 3.
If $1\not\in L_{av}(v_3,\varphi')\setminus \{2,3\}$, then $L_{av}(v_3,\varphi')=\{2,3,5\}$, otherwise we can recolor $v_3$ with a color from $L_{av}(v_3,\varphi')\setminus \{2,3\}$, and the resulted partial coloring satisfies (2). In this case, we recolor $v_3v_4,u_3v_4$ and $u_3v_3$ with 1,2 and 4.
If the current coloring does not satisfy (3), then $L_{av}(v_1v_2,\varphi')=\{2,5\}$, thus we can construct a partial coloring satisfying (3) by recolor $v_3$ with 5. Therefore, we assume that $1\in L_{av}(v_3,\varphi')\setminus \{2,3\}$.
If $5\not\in L_{av}(v_3,\varphi')\setminus \{2,3\}$, then $L_{av}(v_3,\varphi')=\{1,2,3\}$, otherwise we can recolor $v_3$ with a color from $L_{av}(v_3,\varphi')\setminus \{1,2,3\}$, and the resulted partial coloring satisfies (2). In this case, we also recolor $v_3v_4,u_3v_4$ and $u_3v_3$ with 1,2 and 4, and color $v_3$ with 2. If the current coloring does not satisfies (3), then recolor $v_3$ with 3 and one can check that the new coloring satisfies (3). Therefore, $5\in L_{av}(v_3,\varphi')\setminus \{2,3\}$, which implies that $L_{av}(v_3,\varphi')=\{1,3,5\}$. In this case, we recolor $v_3v_4,u_3v_4$ and $u_3v_3$ with 1,2 and 4, and color $v_3$ with 3. If the current coloring does not satisfies (3), then recolor $v_3$ with 5 and the new coloring satisfies (3).

If $5\not\in L_{av}(v_2v_3,\varphi')\setminus \{1,3\}$, then we can do the similar arguments as above by symmetry, so we assume that
$\{4,5\}\subseteq L_{av}(v_2v_3,\varphi')$. Assume that $L_{av}(v_2v_3,\varphi')=\{4,5,n_1,n_2\}$, where $\{n_1,n_2\}\cap \{4,5\}= \emptyset$. If $\{n_1,n_2\}\neq \{1,3\}$, then color $u_3v_3$ with 4. If $5\not\in L_{av}(v_2,\varphi')$, then it is easy to see that the current partial coloring satisfies (1) or (2). If $5\in L_{av}(v_2,\varphi')$, then recolor $u_3v_3$ with 5 and the resulted partial coloring also satisfies (1) or (2). Therefore, we assume that
$L_{av}(v_2v_3,\varphi')=\{1,3,4,5\}$. Now we recolor $v_3v_4$ and $u_3v_4$ with 2 and a color $\phi_2(u_3v_4)\in L_{av}(u_3v_4,\varphi')\setminus \{2\}$, then color $u_3v_3$ with a color $\phi_2(u_3v_3)\in \{4,5\}\setminus \{\phi_2(u_3v_4)\}$. Without loss of generality, assume that $\phi_2(u_3v_3)=4$. We now have
$L_{av}(v_2,\varphi')=\{1,3,5\}$, otherwise $\phi_2$ satisfies (2) and let $\varphi^1:=\phi_2$. If $\phi_2(u_3v_4)\neq 5$, then recolor $u_3v_3$ with 5 and the resulted coloring satisfies (2). If $\phi_2(u_3v_4)=5$, then recolor $u_3v_3$ with 1 and the resulted coloring also satisfies (2).

\noindent\emph{\underline{Case 1.2}. $\{1,2,3\}\not\subseteq L_{av}(u_3v_3,\varphi')$ }

Since $|L_{av}(u_3v_3,\varphi')|=5$, we can assume that $\{4,5,6\}\subseteq L_{av}(u_3v_3,\varphi')$.
If $\{4,5\}\subseteq L_{av}(v_2v_3,\varphi')$, then $L_{av}(v_2v_3,\varphi')=\{1,3,4,5\}$, otherwise we color $u_3v_3$ with 6 and get a partial coloring satisfying (1). We now color $u_3v_3$ with 6, and deduce that $L_{av}(v_2,\varphi')=\{3,4,5\}$, otherwise the current partial coloring satisfies (2). In this case, we recolor $v_3v_4$, $u_3v_4$ and $u_3v_3$ with 2, $\phi_3(u_3v_4)\in L_{av}(u_3v_4,\varphi')\setminus \{2\}$ and $\phi_3(u_3v_3)\in \{4,5,6\}\setminus \{\phi_3(u_3v_4)\}$. It is easy to check that the partial coloring $\phi_3$ satisfies (2), so we let $\varphi':=\phi_3$. By symmetry, one can prove the same result if $\{4,6\}\subseteq L_{av}(v_2v_3,\varphi')$ or $\{5,6\}\subseteq L_{av}(v_2v_3,\varphi')$.
Therefore, we assume, without loss of generality, that $5,6\not\in L_{av}(v_2v_3,\varphi')$. We now color $u_3v_3$ with 5, and deduce that $\{1,3\}\subseteq L_{av}(v_2v_3,\varphi')$, otherwise the current partial coloring satisfies (1). Assume that $L_{av}(v_2v_3,\varphi')=\{1,3,n_1,n_2\}$, where $\{n_1,n_2\}\cap \{1,3\}=\emptyset$. We then have $L_{av}(v_2,\varphi')=\{3,n_1,n_2\}$, because otherwise the current coloring satisfies (2). In this case, we recolor $v_3v_4$, $u_3v_4$ and $u_3v_3$ with 2, $\phi_4(u_3v_4)\in L_{av}(u_3v_4,\varphi')\setminus \{2\}$ and $\phi_4(u_3v_3)\in \{5,6\}\setminus \{\phi_4(u_3v_4)\}$. One can check that the resulted partial coloring $\phi_4$ satisfies (2), thus we let $\varphi':=\phi_4$.


\noindent\emph{\underline{Case 2}. $3\in L_{av}(u_3v_4,\varphi')$}

We first color $u_3v_4$ with 3. Assume that $\{4,5\}\subseteq L_{av}(u_3v_3,\varphi')$.

If $\{4,5\}\subseteq L_{av}(v_2v_3,\varphi')$, then $1\in L_{av}(v_2v_3,\varphi')$. Otherwise, we color $v_3v_4$ and $u_3v_3$ with 1 and 4. If the current coloring does not satisfy (2), then recolor $u_3v_3$ with 5 and get a partial coloring satisfying (2). Similarly, $3\in L_{av}(v_2v_3,\varphi')$, that is, $L_{av}(v_2v_3,\varphi')=\{1,3,4,5\}$. In this case, we color $v_3v_4$ and $u_3v_3$ with 2 and 4. If the current does not satisfy (2), then recolor $u_3v_3$ with 5 and the resulted coloring satisfies (2).
If $\{4,5\}\not\subseteq L_{av}(v_2v_3,\varphi')$, then we assume, without loss of generality, that $4\not\in L_{av}(v_2v_3,\varphi')$. We now color $u_3v_3$ and $v_3v_4$ with 4 and 1. If $\{1,3\}\not\subseteq L_{av}(v_2v_3,\varphi')$, then it is easy to see that the current partial coloring satisfies (1). If
$L_{av}(v_2v_3,\varphi')=\{1,3,n_1,n_2\}$, where $\{n_1,n_2\}\cap \{1,3\}=\emptyset$, then $L_{av}(v_2,\varphi')=\{3,n_1,n_2\}$, otherwise the current partial coloring satisfies (2). In this case, we can also get a partial coloring satisfying (2) by recoloring $v_3v_4$ with 2.

\noindent\emph{\underline{Case 3}. $L_{av}(u_3v_4,\varphi')\cap \{1,2,3\}=\emptyset$.}

Without loss of generality, assume that $L_{av}(u_3v_4,\varphi')=\{4,5\}$. We claim that $1\in L_{av}(v_2v_3,\varphi')$. Otherwise, color $v_3v_4,u_3v_4$ and $u_3v_3$ with 1,4 and $\phi_5(u_3v_3)\in L_{av}(u_3v_3,\varphi')\setminus \{1,3,4\}$. If the current coloring $\phi_5$ does not satisfy (2), then recolor $u_3v_3$ with a color $\phi_6(u_3v_3)\in L_{av}(u_3v_3,\varphi')\setminus \{1,3,4,\phi_5(u_3v_3)\}$. It is easy to see that $\phi_6$ must satisfy (2), thus we let $\varphi':=\phi_6$.
Similarly, $2,3\in L_{av}(v_2v_3,\varphi')$. Assume that $L_{av}(v_2v_3,\varphi')=\{1,2,3,n_1\}$, where $n_1\not\in \{1,2,3\}$. We now color $v_3v_4,u_3v_4$ and $u_3v_3$ with 1,4 and a color $\phi_7(u_3v_3)\in L_{av}(u_3v_3,\varphi')\setminus \{1,2,3,4\}$. If $n_1\neq \phi_7(u_3v_3)$, then we recolor $v_3v_4$ with 2 if $\phi_7$ does not satisfy (2), and the resulted coloring satisfies (2), so $n_1=\phi_7(u_3v_3)$. If $L_{av}(u_3v_3,\varphi')\neq \{1,2,3,4,n_1\}$, then recolor $u_3v_3$ with a color from $L_{av}(u_3v_3,\varphi')\setminus\{1,2,3,4,n_1\}$ and the resulted coloring can be dealt with as above, so $L_{av}(u_3v_3,\varphi')=\{1,2,3,4,n_1\}$ and $n_1\neq 4$. In this case, we color $v_3v_4,u_3v_4$ and $u_3v_3$ with 1,5 and 4. If the current partial coloring does not satisfy (2), then we can construct another partial coloring that satisfies (2) by recoloring $v_3v_4$ with 2.

\vspace{2mm}\noindent \emph{\textbf{Stage 2}. Extend $\varphi^1$ to an $L$-total coloring $\varphi$ of $G$ without altering the assigned colors.}

\vspace{2mm} Note that $\varphi^1$ satisfies $|L_{av}(u_1v_1,\varphi^1)|=|L_{av}(v_1v_2,\varphi^1)|=2$, $|L_{av}(u_1v_2,\varphi^1)|=|L_{av}(u_2v_2,\varphi^1)|=5$ by the choice of $\varphi'$, and $|L_{av}(u_2v_2,\varphi^1)|,|L_{av}(v_2v_3,\varphi^1)|,|L_{av}(v_2,\varphi^1)|\geq 2$, and moreover, either
$L_{av}(v_2v_3,\varphi^1)\neq L_{av}(v_2,\varphi^1)$
or $L_{av}(v_1v_2,\varphi^1)\neq L_{av}(v_2,\varphi^1)$ if $|L_{av}(v_2,\varphi^1)|=2$, by the choice of $\varphi^1$ in Stage 1.

Without loss of generality, we assume that $|L_{av}(u_2v_2,\varphi^1)|=|L_{av}(v_2v_3,\varphi^1)|=|L_{av}(v_2,\varphi^1)|=2$ and $L_{av}(v_2v_3,\varphi^1)\neq L_{av}(v_2,\varphi^1)=\{1,2\}$ in the following arguments.

If $L_{av}(v_1v_2,\varphi^1)\cap L_{av}(u_2v_3,\varphi^1)\neq \emptyset$, then color $v_1v_2$ and $u_2v_3$ with $\mu(v_1v_2)=\mu(u_2v_3)\in L_{av}(v_1v_2,\varphi^1)\cap L_{av}(u_2v_3,\varphi^1)$. Since $L_{av}(v_2v_3,\varphi^1)\neq L_{av}(v_2,\varphi^1)$, we can color $v_2$ and $v_2v_3$ with $\mu(v_2)\in L_{av}(v_2,\varphi^1)\setminus \{\mu(v_1v_2)\}$ and $\mu(u_2v_3)\in L_{av}(u_2v_3,\varphi^1)\setminus \{\mu(u_2v_3)\}$ so that $\mu(v_2)\neq \mu(u_2v_3)$. We then color $u_1v_1,u_1v_2$ and $u_2v_2$ with $\mu(u_1v_1)\in L_{av}(u_1v_1,\varphi^1)\setminus \{\mu(v_1v_2)\}$, $\mu(u_1v_2)\in L_{av}(u_1v_2,\varphi^1)\setminus \{\mu(u_1v_1),\mu(v_1v_2),\mu(v_2),\mu(v_2v_3)\}$ and $\mu(u_2v_2)\in L_{av}(u_2v_2,\varphi^1)\setminus \{\mu(u_2v_3),\mu(v_2v_3),\mu(v_2),\mu(u_1v_2)\}$, respectively. Since $u_1,u_2$ and $u_3$ are 2-vertices, they are easily colored at the last stage. Therefore, we have an $L$-total coloring $\mu$ of $G$. In what follows, we assume that
$L_{av}(v_1v_2,\varphi^1)\cap L_{av}(u_2v_3,\varphi^1)=\emptyset$.

\noindent\emph{\underline{Case 1$'$}. $L_{av}(v_2v_3,\varphi^1)=\{1,3\}$ }

If $1\in L_{av}(u_1v_1,\varphi^1)$ and $L_{av}(v_1v_2,\varphi^1)\neq \{1,2\}$, then color $u_1v_1$ and $v_2v_3$ with 1, and color $v_2,v_1v_2,u_2v_3,u_2v_2$ and $u_1v_2$ with 2, $\mu_1(v_1v_2)\in L_{av}(v_1v_2,\varphi^1)\setminus \{1,2\}$,
$\mu_1(u_2v_3)\in L_{av}(u_2v_3,\varphi^1)\setminus \{1\}$, $\mu_1(u_2v_2)\in L_{av}(u_2v_2,\varphi^1)\setminus \{1,2,\mu_1(v_1v_2),\mu_1(u_2v_3)\}$ and
 $\mu_1(u_1v_2)\in L_{av}(u_1v_2,\varphi^1)\setminus \{1,2,\mu_1(v_1v_2),\mu_1(u_2v_2)\}$, respectively.
If $1\in L_{av}(u_1v_1,\varphi^1)$ and $L_{av}(v_1v_2,\varphi^1)= \{1,2\}$, then color $u_1v_1$ and $v_2$ with 1, and color $v_2v_3,v_1v_2,u_2v_3,u_2v_2$ and $u_1v_2$ with 3, 2,
$\mu_1(u_2v_3)\in L_{av}(u_2v_3,\varphi^1)\setminus \{3\}$, $\mu_1(u_2v_2)\in L_{av}(u_2v_2,\varphi^1)\setminus \{1,2,3,\mu_1(u_2v_3)\}$ and
$\mu_1(u_1v_2)\in L_{av}(u_1v_2,\varphi^1)\setminus \{1,2,3,\mu_1(u_2v_2)\}$, respectively.  In each case, we can extend $\mu_1$ to the 2-vertices $u_1,u_2$ and $u_3$ and get an $L$-total coloring of $G$. Therefore, $1\not\in L_{av}(u_1v_1,\varphi^1)$. Similarly, $2,3\not\in L_{av}(u_1v_1,\varphi^1)$. We assume, without loss of generality, that $L_{av}(u_1v_1,\varphi^1)=\{4,5\}$

If $4\not\in L_{av}(u_1v_2,\varphi^1)$, then color $u_1v_1$ and $v_1v_2$ with 4 and $\mu_2(v_1v_2)\in L_{av}(v_1v_2,\varphi^1)\setminus \{4\}$. If $\mu_2(v_1v_2)\neq 2$, then color $v_2,v_2v_3,u_2v_3,u_2v_2$ and $u_1v_2$ with 2, $\mu_2(v_2v_3)\in \{1,3\}\setminus \{\mu_2(v_1v_2)\}$, $\mu_2(u_2v_3)\in L_{av}(u_2v_3,\varphi^1)\setminus \{\mu_2(v_2v_3)\}$, $\mu_2(u_2v_2)\in L_{av}(u_2v_2,\varphi^1)\setminus \{2,\mu_2(v_1v_2),\mu_2(v_2v_3),\mu_2(u_2v_3)\}$ and
$\mu_2(u_1v_2)\in L_{av}(u_1v_2,\varphi^1)\setminus \{2,\mu_2(v_1v_2),\mu_2(v_2v_3),\mu_2(u_2v_2)\}$.
If $\mu_2(v_1v_2)=2$, then color $v_2,v_2v_3,u_2v_3,u_2v_2$ and $u_1v_2$ with 1, 3, $\mu_2(u_2v_3)\in L_{av}(u_2v_3,\varphi^1)\setminus \{3\}$, $\mu_2(u_2v_2)\in L_{av}(u_2v_2,\varphi^1)\setminus \{1,2,3,\mu_2(u_2v_3)\}$ and
$\mu_2(u_1v_2)\in L_{av}(u_1v_2,\varphi^1)\setminus \{1,2,3,\mu_2(u_2v_2)\}$.
In each case, we can extend $\mu_2$ to the 2-vertices $u_1,u_2$ and $u_3$ and get an $L$-total coloring of $G$. Therefore, $4\in L_{av}(u_1v_2,\varphi^1)$. Similarly, we have $1,2,3,5\in L_{av}(u_1v_2,\varphi^1)$, that is, $L_{av}(u_1v_2,\varphi^1)=\{1,2,3,4,5\}$. By similar arguments as above, we can also prove that $L_{av}(u_2v_2,\varphi^1)=\{1,2,3,4,5\}$, $L_{av}(v_1v_2,\varphi^1)\subseteq \{1,2,3,4,5\}$ and $L_{av}(u_2v_3,\varphi^1)\subseteq \{1,2,3,4,5\}$.

If $1\in L_{av}(v_1v_2,\varphi^1)$, then color $v_1v_2,v_2,v_2v_3,u_1v_1,u_1v_2$ and $u_2v_2$ with 1,2,3,4,5 and 4. If $L_{av}(u_2v_3,\varphi^1)\neq \{3,4\}$, then color $u_2v_3$ with a color in $L_{av}(u_2v_3,\varphi^1)\setminus\{3,4\}$. If $L_{av}(u_2v_3,\varphi^1)=\{3,4\}$, then recolor $u_1v_1,u_1v_2$ and $u_2v_2$ with 5,4 and 5, and color $u_2v_3$ with 4. In each case, we can extend the current coloring to the 2-vertices $u_1,u_2$ and $u_3$ and get an $L$-total coloring of $G$. By similar arguments as above, we can complete the proof of this lemma if $2\in L_{av}(v_1v_2,\varphi^1)$ or $3\in L_{av}(v_1v_2,\varphi^1)$. Therefore, we assume that $L_{av}(v_1v_2,\varphi^1)=\{4,5\}$. In this case, we color $v_1v_2,v_2,v_2v_3,u_1v_1,u_1v_2$ and $u_2v_2$ with 4,1,3,5,2 and 5. Since $L_{av}(v_1v_2,\varphi^1)\cap L_{av}(u_2v_3,\varphi^1)=\emptyset$, $5\not\in L_{av}(u_2v_3,\varphi^1)$. Hence, we color $u_2v_3$ with a color from
$L_{av}(u_2v_3,\varphi^1)\setminus \{3\}$ and then extend the coloring at this stage to $u_1,u_2$ and $u_3$ to obtain an $L$-total coloring of $G$.

\noindent\emph{\underline{Case 2$'$}. $L_{av}(v_2v_3,\varphi^1)=\{3,4\}$. }

By similar arguments as in the first part of Case 1$'$, one can prove that $L_{av}(u_1v_1,\varphi^1)\cap\{1,2,3,4\}=\emptyset$, so we assume that
$L_{av}(u_1v_1,\varphi^1)=\{5,6\}$. Since $|L_{av}(u_1v_2,\varphi^1)|=5$, $\{1,2,3,4,5,6\}\setminus L_{av}(u_1v_2,\varphi^1)\neq \emptyset$. Without loss of generality, assume that $1\not\in L_{av}(u_1v_2,\varphi^1)$ and color $v_2$ with 1.
If $L_{av}(v_1v_2,\varphi^1)=\{1,3\}$, then color $v_2v_3,v_1v_2,u_1v_1,u_2v_3,u_2v_2$ and $u_1v_2$
with 4, 3, 5, $\mu_3(u_2v_3)\in L_{av}(u_2v_3,\varphi^1)\setminus \{4\}$, $\mu_3(u_2v_2)\in L_{av}(u_2v_2,\varphi^1)\setminus \{1,3,4,\mu_3(u_2v_3)\}$ and
$\mu_3(u_1v_2)\in L_{av}(u_1v_2,\varphi^1)\setminus \{3,4,5,\mu_3(u_2v_2)\}$.
If $L_{av}(v_1v_2,\varphi^1)\neq \{1,3\}$, then color $v_2v_3,v_1v_2,u_1v_1,u_2v_3,u_2v_2$ and $u_1v_2$
with 3, $\mu_3(v_1v_2)\in L_{av}(v_1v_2,\varphi^1)\setminus \{1,3\}$, 5, $\mu_3(u_2v_3)\in L_{av}(u_2v_3,\varphi^1)\setminus \{3\}$, $\mu_3(u_2v_2)\in L_{av}(u_2v_2,\varphi^1)\setminus \{1,3,\mu_3(v_1v_2),\mu_3(u_2v_3)\}$ and
$\mu_3(u_1v_2)\in L_{av}(u_1v_2,\varphi^1)\setminus \{3,5,\mu_3(v_1v_2),\mu_3(u_2v_2)\}$. In each case, we can extend the partial coloring $\mu_3$ to the 2-vertices $u_1,u_2$ and $u_3$ and get an $L$-total coloring of $G$.
\end{proof}

\begin{lem}\label{de}
Suppose that $G$ contains a cycle $u_1u_2u_3u_4$ with $u_2u_4\in E(G)$ and $d(u_2)=d(u_4)=3$ and that
$G'=G-\{u_2,u_4\}$ has an $L'$-total coloring $\varphi'$ so that $L'(x)=L(x)$ for each $x\in VE(G')$, where $L$ is a list assignment of $G$.
If $$|L_{av}(u_2u_4,\varphi')|\geq 6, $$
$$\min\{|L_{av}(u_2,\varphi')|,|L_{av}(u_4,\varphi')|\}\geq 4,$$ $$\min\{|L_{av}(u_1u_2,\varphi')|,|L_{av}(u_2u_3,\varphi')|,|L_{av}(u_3u_4,\varphi')|,|L_{av}(u_1u_4,\varphi')|\}\geq 2$$
and $$L_{av}(u_1u_2,\varphi')\neq L_{av}(u_2u_3,\varphi')~when~|L_{av}(u_1u_2,\varphi')|=|L_{av}(u_2u_3,\varphi')|=2$$
then $\varphi'$ can be extended to an $L$-total coloring $\varphi$ of $G$ without altering the colors in $G'$.
\end{lem}

\begin{proof}
Without loss of generality, we assume that $|L_{av}(u_2u_4,\varphi')|=6$, $|L_{av}(u_2,\varphi')|=|L_{av}(u_4,\varphi')|=4$ and $|L_{av}(u_1u_2,\varphi')|=|L_{av}(u_2u_3,\varphi')|=|L_{av}(u_3u_4,\varphi')|=|L_{av}(u_1u_4,\varphi')|=2$ (otherwise we can shorten some lists that assigned to the elements of $VE(G)$ so that those conditions are satisfied), and then
split the proofs into the following two cases.\\[.2em]
                       \noindent \emph{\underline{Case 1}. $L_{av}(u_1u_2,\varphi')=\{1,2\}$ and $L_{av}(u_2u_3,\varphi')=\{3,4\}$.}\\[.2em]
\indent If $L_{av}(u_1u_4,\varphi')=\{1,2\}$, then color $u_1u_2$ and $u_1u_4$ with 1 and 2, $u_3u_4$ with $\varphi''(u_3u_4)\in L_{av}(u_3u_4,\varphi')\setminus \{2\}$ and $u_2u_3$ with $\varphi''(u_2u_3)\in \{3,4\}\setminus \{\varphi''(u_3u_4)\}$. Denote the extended partial coloring by $\varphi''$.
One can see that $|L_{av}(u_2,\varphi'')|, |L_{av}(u_4,\varphi'')|, |L_{av}(u_2u_4,\varphi'')|\geq 2$. If $L_{av}(u_2,\varphi'')=L_{av}(u_4,\varphi'')=L_{av}(u_2u_4,\varphi'')$, then recolor $u_1u_2$ and $u_1u_4$ with 2 and 1 when $\varphi''(u_3u_3)\neq 1$, or recolor $u_2u_3$ with the color in $\{3,4\}\setminus \{\varphi''(u_2u_3)\}$ when $\varphi''(u_3u_3)=1$. We still the denote current coloring by $\varphi''$ but now we do not have $L_{av}(u_2,\varphi'')=L_{av}(u_4,\varphi'')=L_{av}(u_2u_4,\varphi'')$. Therefore, we can easily extend $\varphi''$ to an $L$-total coloring $\varphi$ of $G$ by coloring $u_2,u_4$ and $u_2u_4$ properly.

If $L_{av}(u_1u_4,\varphi')\cap L_{av}(u_2u_3,\varphi')\neq \emptyset$, then color $u_2u_3$ and $u_1u_4$ with $\varphi''(u_2u_3)=\varphi''(u_1u_4)\in L_{av}(u_1u_4,\varphi')\cap L_{av}(u_2u_3,\varphi')$, $u_1u_2$ with $\varphi''(u_1u_2)\in L_{av}(u_1u_2,\varphi')\setminus \{\varphi''(u_2u_3)\}$, $u_3u_4$ with $\varphi''(u_3u_4)\in L_{av}(u_3u_4,\varphi')\setminus \{\varphi''(u_2u_3)\}$ and denote the extended coloring by $\varphi''$. One can see that $|L_{av}(u_2,\varphi'')|,$\\$|L_{av}(u_4,\varphi'')|\geq 2$ and $|L_{av}(u_2u_4,\varphi'')|\geq 3$. Therefore, $\varphi''$ can be easily extended to an $L$-total coloring $\varphi$ of $G$ by coloring $u_1u_2,u_2,u_2u_4$ and $u_4$ properly.

If $L_{av}(u_1u_4,\varphi')\cap L_{av}(u_2u_3,\varphi')=\emptyset$ and $L_{av}(u_1u_4,\varphi')\not=\{1,2\}$, then color $u_1u_4$ with $\varphi''(u_1u_4)\in L(u_1u_4,\varphi')\setminus \{1,2\}$, $u_3u_4$ with $\varphi''(u_3u_4)\in L(u_3u_4,\varphi')\setminus \{\varphi''(u_1u_4)\}$, $u_2u_3$ with $\varphi''(u_2u_3)\in L(u_2u_3,\varphi')\setminus \{\varphi''(u_3u_4)\}$ and denote the extended coloring by $\varphi''$. One can see that $|L_{av}(u_1u_2,\varphi'')|,|L_{av}(u_4,\varphi'')|\geq 2$ and $|L_{av}(u_2,\varphi'')|,|L_{av}(u_2u_4,\varphi'')|\geq 3$. Therefore, $\varphi''$ can be easily extended to an $L$-total coloring $\varphi$ of $G$ by coloring $u_1u_2,u_2,u_2u_4$ and $u_4$ properly.\\[.2em]
                       \noindent \emph{\underline{Case 2}. $L_{av}(u_1u_2,\varphi')=\{1,2\}$ and $L_{av}(u_2u_3,\varphi')=\{1,3\}$.}\\[.2em]
\indent By similar arguments as in the second part of Case 1, one can show that $L_{av}(u_1u_4,\varphi')\cap L_{av}(u_2u_3,\varphi')=\emptyset$ and $L_{av}(u_1u_2,\varphi')\cap L_{av}(u_3u_4,\varphi')=\emptyset$.

If $2\in L_{av}(u_1u_4,\varphi')$, then we assume that $L_{av}(u_1u_4,\varphi')=\{2,4\}$. If $L_{av}(u_3u_4,\varphi')\neq \{3,4\}$, then color $u_2u_3$ and $u_1u_4$ with $3$ and $4$, $u_3u_4$ with $\varphi''(u_3u_4)\in L_{av}(u_3u_4,\varphi')\setminus\{3,4\}$ and denote the extended coloring by $\varphi$.
One can see that $|L_{av}(u_1u_2,\varphi'')|,|L_{av}(u_4,\varphi'')|\geq 2$ and $|L_{av}(u_2,\varphi'')|,|L_{av}(u_2u_4,\varphi'')|\geq 3$.  Therefore, $\varphi''$ can be easily extended to an $L$-total coloring $\varphi$ of $G$ by coloring $u_1u_2,u_2,u_2u_4$ and $u_4$ properly. If $L_{av}(u_3u_4,\varphi')= \{3,4\}$, then we first color $u_1u_2,u_2u_3,u_3u_4$ and $u_1u_4$ with 1, 3, 4 and 2, and denote the extended coloring by $\varphi''$. It is easy to see that
$|L_{av}(u_2,\varphi'')|,|L_{av}(u_4,\varphi'')|,|L_{av}(u_2u_4,\varphi'')|\geq 2$. If the three sets
$L_{av}(u_2,\varphi''),L_{av}(u_4,\varphi''),L_{av}(u_2u_4,\varphi'')$ are not the same or $L_{av}(u_2,\varphi'')=L_{av}(u_4,\varphi'')=L_{av}(u_2u_4,\varphi'')$ and $|L_{av}(u_2,\varphi'')|\geq 3$, then $\varphi''$ can be easily extended to an $L$-total coloring $\varphi$ of $G$. If $L_{av}(u_2,\varphi'')=L_{av}(u_4,\varphi'')=L_{av}(u_2u_4,\varphi'')=\{5,6\}$, then we revise the coloring $\varphi''$ by recoloring $u_1u_2,u_2u_3,u_3u_4$ and $u_4u_1$ by 2, 1, 3 and 4. After that, we have $L_{av}(u_2,\varphi'')=\{3,5,6\}$, $L_{av}(u_4,\varphi'')=\{2,5,6\}$ and $L_{av}(u_2u_4,\varphi'')=\{5,6\}$, so we extend $\varphi''$ to an $L$-total coloring of $G$ by coloring $u_2,u_4$ and $u_2u_4$ with $3,2$ and $5$.

If $2\not\in L_{av}(u_1u_4,\varphi')$, then we assume that $L_{av}(u_1u_4,\varphi')=\{4,5\}$. We now color $u_2u_3$ with 3, $u_3u_4$ with $\varphi''(u_3u_4)\in L_{av}(u_3u_4,\varphi'')\setminus \{3\}$, $u_1u_4$ with $\varphi''(u_1u_4)\in L_{av}(u_1u_4,\varphi'')\setminus \{\varphi''(u_3u_4)\}$ and denote the extended coloring by $\varphi''$. It is easy see that $|L_{av}(u_1u_2,\varphi'')|,|L_{av}(u_4,\varphi'')|\geq 2$ and $|L_{av}(u_2,\varphi'')|,|L_{av}(u_2u_4,\varphi'')|\geq 3$.  Therefore, $\varphi''$ can be easily extended to an $L$-total coloring $\varphi$ of $G$ by coloring $u_1u_2,u_2,u_2u_4$ and $u_4$ properly.
\end{proof}

\begin{lem}\label{f}
Suppose that $G$ contains a cycle $u_1u_2u_3u_4$ so that $u_2u_4\in E(G)$, $d(u_2)=d(u_4)=3$ and $u_3$ is adjacent to a vertex $v$. Let
$\varphi'$ be a partial $L$-total coloring of $G$ so that the uncolored elements under $\varphi'$ are $u_1u_2,u_2u_3,u_3u_4,u_1u_4,u_2u_4,u_3v,u_2,u_3,u_4$ and $v$, where $L$ is a list assignment of $G$. If
$$|L_{av}(u_2u_4,\varphi')|\geq 6,$$
$$\min\{|L_{av}(u_2,\varphi')|,|L_{av}(u_4,\varphi')|\}\geq 5,$$
$$\min\{|L_{av}(u_2u_3,\varphi')|,|L_{av}(u_3u_4,\varphi')|\}\geq 4$$ $$\min\{|L_{av}(u_1u_2,\varphi')|,|L_{av}(u_1u_4,\varphi')|,|L_{av}(u_3v,\varphi')|,|L_{av}(u_3,\varphi')|,|L_{av}(v,\varphi')|\}\geq 2$$
and $$at~least~two~of~L_{av}(u_3v,\varphi'),L_{av}(u_3,\varphi'),L_{av}(v,\varphi')~are~distinct~when$$
$$|L_{av}(u_3v,\varphi')|=|L_{av}(u_3,\varphi')|=|L_{av}(v,\varphi')|=2,$$
then $\varphi'$ can be extended to an $L$-total coloring $\varphi$ of $G$ without altering the assigned colors.
\end{lem}

\begin{proof}
Without loss of generality, we assume that
$|L_{av}(u_2u_4,\varphi')|=6,$
$|L_{av}(u_2,\varphi')|=|L_{av}(u_4,\varphi')|=5,$
$|L_{av}(u_2u_3,\varphi')|=|L_{av}(u_3u_4,\varphi')|=4$ and $|L_{av}(u_1u_2,\varphi')|=|L_{av}(u_1u_4,\varphi')|=|L_{av}(u_3v,\varphi')|=|L_{av}(u_3,\varphi')|=|L_{av}(v,\varphi')|=2.$

If $L_{av}(u_3v,\varphi')=L_{av}(u_3,\varphi')=\{1,2\}$, then we color $u_3$ and $u_3v$ with $1$ and $2$, and then color $v$ with a color from $L_{av}(v,\varphi')$ that is different with 1 and 2. Denote the current partial coloring still by $\varphi'$. We then have
$|L_{av}(u_2u_4,\varphi')|=6, |L_{av}(u_1u_2,\varphi')|=|L_{av}(u_1u_4,\varphi')|=2$, $|L_{av}(u_2,\varphi')|, |L_{av}(u_4,\varphi')|\geq 4$ and $|L_{av}(u_2u_3,\varphi')|,|L_{av}(u_3u_4,\varphi')|\geq 2$. Without loss of generality, assume that $|L_{av}(u_2,\varphi')|=|L_{av}(u_4,\varphi')|=4$ and $|L_{av}(u_2u_3,\varphi')|=|L_{av}(u_3u_4,\varphi')|=2$.
Since every 4-cycle is 2-choosable, we color each edge of the cycle $u_1u_2u_3u_4$ from its available list and denote the coloring at this stage by $\varphi''$. It is easy to see that
$|L_{av}(u_2,\varphi'')|,|L_{av}(u_4,\varphi'')|,|L_{av}(u_2u_4,\varphi'')|\geq 2$. If
at least two of $L_{av}(u_2,\varphi''), L_{av}(u_4,\varphi'')$ and $L_{av}(u_2u_4,\varphi'')$ are distinct or $\max\{|L_{av}(u_2,\varphi'')|,|L_{av}(u_4,\varphi'')|,|L_{av}(u_2u_4,\varphi'')|\}\geq 3$, then $\varphi''$ can be easily extended to an $L$-total coloring of $G$. If
$L_{av}(u_2,\varphi'')=L_{av}(u_4,\varphi'')=L_{av}(u_2u_4,\varphi'')=2$, then exchange the colors on $u_3$ and $u_3v$ and denote this coloring by $\varphi'''$. This operation does not disturb the properness of the colors on the edges of the cycle $u_1u_2u_3u_4$, but implies that
at least two of $L_{av}(u_2,\varphi'''), L_{av}(u_4,\varphi''')$ and $L_{av}(u_2u_4,\varphi''')$ are distinct if $|L_{av}(u_2,\varphi''')|=|L_{av}(u_4,\varphi''')|=|L_{av}(u_2u_4,\varphi''')|=2$. Therefore, $\varphi'''$ can be extended to an $L$-total coloring $\varphi$ of $G$.

If $L_{av}(u_3v,\varphi')=\{1,2\}$ and $L_{av}(u_3,\varphi')=\{1,3\}$, then there are two ways to color $u_3$ and $u_3v$ so that $v$ can be colored from its available list so that the color assigned to $v$ is different with the colors assigned to $u_3$ and $u_3v$. Without loss of generality, assume the above two ways of coloring are as follows: color $u_3$ and $u_3v$ with 1 and 3, or with 2 and 1. We now color $u_3$ and $u_3v$ with 1 and 3, and then color $v$ properly.  Denote the current coloring by $\varphi''$. Suppose that $L_{av}(u_1u_2,\varphi')=\{a,b\}$. If $L_{av}(u_2u_3,\varphi')\not=\{1,3,a,b\}$ or $\{a,b\}\cap \{1,3\}\neq \emptyset$, then $L_{av}(u_1u_2,\varphi'')\neq L_{av}(u_2u_3,\varphi'')$, therefore, by Lemma \ref{de}, $\varphi''$ can be extended to an $L$-total coloring of $G$. If
$L_{av}(u_2u_3,\varphi')=\{1,3,a,b\}$ and $\{a,b\}\cap \{1,3\}=\emptyset$, then recolor $v_3$ and $v$ by 2 and 1, and recolor $v$ properly. Denote this coloring by $\varphi'''$. We then have $L_{av}(u_2u_3,\varphi''')=\{3,a,b\}\setminus \{2\}\neq \{a,b\}=L_{av}(u_1u_2,\varphi''')$, so by Lemma \ref{de}, $\varphi'''$ can be extended to an $L$-total coloring of $G$.

If $L_{av}(u_3v,\varphi')=\{1,2\}$ and $L_{av}(u_3,\varphi')=\{3,4\}$, then there are two ways to color $u_3$ and $u_3v$ so that $v$ can be colored from its available list so that the color assigned to $v$ is different with the colors assigned to $u_3$ and $u_3v$. Therefore, we can do similar arguments as above to complete the proof.
\end{proof}

\begin{lem}\label{g}
Suppose that $G$ contains a cycle $u_1u_2u_3u_4$ so that $u_1u_3,u_2u_4\in E(G)$ and $d(u_2)=d(u_4)=3$. Let
$\varphi'$ be a partial $L$-total coloring of $G$ so that the uncolored elements under $\varphi'$ are $u_1u_2,u_2u_3,u_3u_4,u_1u_4,u_2u_4,u_1u_3,u_1,u_2,u_3$ and $u_4$, where $L$ is a list assignment of $G$. If
$$\min\{|L_{av}(u_2,\varphi')|,|L_{av}(u_4,\varphi')|,|L_{av}(u_2u_4,\varphi')|\}\geq 6,$$
$$\min\{|L_{av}(u_1u_2,\varphi')|,|L_{av}(u_2u_3,\varphi')|,|L_{av}(u_3u_4,\varphi')|,|L_{av}(u_1u_4,\varphi')|\}\geq 4$$ $$\min\{|L_{av}(u_1,\varphi')|,|L_{av}(u_3,\varphi')|,|L_{av}(u_1u_3,\varphi')|\}\geq 2$$
and $$at~least~two~of~L_{av}(u_1,\varphi'),L_{av}(u_3,\varphi'),L_{av}(u_1u_3,\varphi')~are~distinct~when$$
$$|L_{av}(u_1,\varphi')|=|L_{av}(u_3,\varphi')|=|L_{av}(u_1u_3,\varphi')|=2,$$
then $\varphi'$ can be extended to an $L$-total coloring $\varphi$ of $G$ without altering the assigned colors.
\end{lem}

\begin{proof}
Without loss of generality, we assume that
$|L_{av}(u_2,\varphi')|=|L_{av}(u_4,\varphi')|=|L_{av}(u_2u_4,\varphi')|=6,$
$|L_{av}(u_1u_2,\varphi')|=|L_{av}(u_2u_3,\varphi')|=|L_{av}(u_3u_4,\varphi')|=|L_{av}(u_1u_4,\varphi')|=4$ and $|L_{av}(u_1,\varphi')|=|L_{av}(u_3,\varphi')|=|L_{av}(u_1u_3,\varphi')|=2$.

If $L_{av}(u_1u_3,\varphi')=L_{av}(u_3,\varphi')=\{1,2\}$, then color $u_3$ and $u_1u_3$ with $1$ and $2$, $u_1$ with $\varphi''(u_1)\in L_{av}(u_1,\varphi')\setminus \{1,2\}\neq \emptyset$ and denote the current coloring by $\varphi''$. Without loss of generality, assume that $\varphi''(u_1)=3$.
It is easy to see that
$|L_{av}(u_2u_4,\varphi'')|=6$,
$|L_{av}(u_2,\varphi'')|,|L_{av}(u_4,\varphi'')|\geq 4$ and
$|L_{av}(u_1u_2,\varphi'')|,|L_{av}(u_2u_3,\varphi'')|,|L_{av}(u_3u_4,\varphi'')|,|L_{av}(u_1u_4,\varphi'')|\geq 2$. If $\{1,2\}\not\subseteq L_{av}(u_2u_3,\varphi')$, or $\{2,3\}\not\subseteq L_{av}(u_1u_2,\varphi')$, or $L_{av}(u_2u_3,\varphi')\setminus \{1,2\}\neq L_{av}(u_1u_2,\varphi')\setminus \{2,3\}$, then by Lemma \ref{de}, $\varphi''$ can be extended to an $L$-total coloring of $G$. If $L_{av}(u_1u_2,\varphi')=\{2,3,a,b\}$,
$L_{av}(u_2u_3,\varphi')=\{1,2,a,b\}$ and $\{a,b\}\cap \{1,2,3\}\neq \emptyset$, then exchange the colors on $u_3$ and $u_1u_3$, and denote this coloring by $\varphi'''$. Since $|L_{av}(u_2u_4,\varphi''')|=6$,
$|L_{av}(u_2,\varphi''')|,|L_{av}(u_4,\varphi''')|\geq 4$,
$|L_{av}(u_1u_2,\varphi''')|\geq 3$ and $|L_{av}(u_3u_4,\varphi''')|,$\\$|L_{av}(u_1u_4,\varphi''')|,|L_{av}(u_2u_3,\varphi''')|\geq 2$, by Lemma \ref{de}, $\varphi'''$ can be extended to an $L$-total coloring $\varphi$ of $G$.

If $L_{av}(u_1u_3,\varphi')=\{1,3\}$ and $L_{av}(u_3,\varphi')=\{1,2\}$, then we shall assume that $L_{av}(u_1,\varphi')\neq \{1,3\}$ (otherwise we come back to the above case).
If $L_{av}(u_1,\varphi')\setminus \{1,2,3\}\neq \emptyset$, then color $u_1$ with a color in $L_{av}(u_1,\varphi')\setminus \{1,2,3\}$, say 4, and color $u_3$ and $u_1u_3$ with 1 and 3. Denote the current coloring by $\varphi''$. If $\{1,3\}\not\subseteq L_{av}(u_2u_3,\varphi')$, or $\{3,4\}\not\subseteq L_{av}(u_1u_2,\varphi')$, or $L_{av}(u_2u_3,\varphi')\setminus \{1,3\}\neq L_{av}(u_1u_2,\varphi')\setminus \{3,4\}$, then by Lemma \ref{de} and similar arguments as before, $\varphi''$ can be extended to an $L$-total coloring of $G$.
If $L_{av}(u_1u_2,\varphi')=\{3,4,a,b\}$,
$L_{av}(u_2u_3,\varphi')=\{1,3,a,b\}$ and $\{a,b\}\cap \{1,3,4\}\neq \emptyset$, then recolor $u_3$ with 2 and denote the current coloring by $\varphi'''$.
By Lemma \ref{de} and similar arguments as before, one can prove that $\varphi'''$ can be extended to an $L$-total coloring of $G$.
If $L_{av}(u_1,\varphi')=\{1,2\}$, then color $u_3,u_1u_3$ and $u_1$ with 1,3 and 2. Denote this partial coloring of $G$ by $\varphi''$. By similar arguments as above, one can prove that either $\varphi''$ can be extended to an $L$-total coloring of $G$ by Lemma \ref{de}, or we can construct a new partial coloring $\varphi'''$ by exchanging the colors on $u_1$ and $u_3$ that can be extended to an $L$-total coloring of $G$ by Lemma \ref{de}.
If $L_{av}(u_1,\varphi')=\{2,3\}$, then color $u_3,u_1u_3$ and $u_1$ with 1,3 and 2 and denote this partial coloring of $G$ by $\varphi''$. One can also prove that either $\varphi''$ can be extended to an $L$-total coloring of $G$ by Lemma \ref{de}, or we can construct a new partial coloring $\varphi'''$ that can be extended to an $L$-total coloring of $G$ by Lemma \ref{de} via recoloring $u_3,u_1u_3$ and $u_1$ with 2,1 and 3.

If $L_{av}(u_1u_3,\varphi')=\{3,4\}$ and $L_{av}(u_3,\varphi')=\{1,2\}$, then we shall assume that $L_{av}(u_1,\varphi')\cap \{3,4\}=\emptyset$. (otherwise we come back to the above case). We now
color $u_1$ with $\varphi''(u_1)\in L_{av}(u_1,\varphi')$, $u_1u_3$ with $\varphi''(u_1u_3)\in L_{av}(u_1u_3,\varphi')\setminus \{\varphi''(u_1)\}$, and $u_3$ with $\varphi''(u_3)\in L_{av}(u_3,\varphi')\setminus \{\varphi''(u_1)\}$.
Denote this coloring by $\varphi''$.
By Lemma \ref{de} and similar arguments as before, $\varphi''$ can be extended to an $L$-total coloring of $G$ if
$\{\varphi''(u_1u_3),\varphi''(u_3)\}\not\subseteq L_{av}(u_2u_3,\varphi')$, or $\{\varphi''(u_1u_3),\varphi''(u_1)\}\not\subseteq L_{av}(u_1u_2,\varphi')$, or $L_{av}(u_2u_3,\varphi')\setminus \{\varphi''(u_1u_3),\varphi''(u_3)\}\neq L_{av}(u_1u_2,\varphi')\setminus \{\varphi''(u_1u_3),\varphi''(u_1)\}$.
Therefore, we assume that $L_{av}(u_2u_3,\varphi')=\{\varphi''(u_1u_3),\varphi''(u_3),a,b\}$ and $L_{av}(u_1u_2,\varphi')= \{\varphi''(u_1u_3),\varphi''(u_1),a,b\}$, where $\{a,b\}\cap \{\varphi''(u_1u_3),\varphi''(u_1),\varphi''(u_3)\}=\emptyset$. If $\varphi''(u_1)\not\in \{1,2\}$, then we can construct a new partial coloring $\varphi'''$ that can be extended to an $L$-total coloring of $G$ by Lemma \ref{de} via recoloring $u_3$ with a color form $L_{av}(u_3,\varphi')\setminus \{\varphi''(u_3)\}$. Thus, we assume, without loss of generality, that $\varphi''(u_1)=1$, $\varphi''(u_3)=2$ and $\varphi''(u_1u_3)=3$. If $L_{av}(u_1,\varphi')\neq \{1,2\}$, then recolor $u_1$ with a color from $L_{av}(u_1,\varphi')\setminus \{1,2,3,4\}$, which is a non-empty set since $L_{av}(u_1,\varphi')\cap \{3,4\}=\emptyset$. If $L_{av}(u_1,\varphi')=\{1,2\}$, then exchange the colors on $u_1$ and $u_3$. In each case, the resulted coloring can be extended to an $L$-total coloring of $G$ by Lemma \ref{de}.
\end{proof}

We are now ready to give a proof of Theorem \ref{listtotal}.

\begin{proof}[\textbf{Proof of Theorem \ref{listtotal}}]

Suppose that $G$ is a counterexample to this theorem with the smallest value of $|V(G)|+|E(G)|$. It is easy to see that $\delta(G)\geq 2$ and every 2-vertex in $G$ is adjacent only to $\Delta(G)$-vertices, so by Theorem \ref{opg}, $G$ contains one of the configurations among (b)--(g).

If $G$ contains $(b)$, then $G'=G-\{u_1,u_2,u_3,v_2,v_3\}$ has an $L'$-total coloring $\varphi'$ so that $L'(x)=L(x)$ for each $x\in VE(G')$. One can see that
$|L_{av}(v_1v_2,\varphi')|,|L_{av}(v_1v_3,\varphi')|,|L_{av}(v_2v_4,\varphi')|,|L_{av}(v_3v_4,\varphi')|\geq 3$. If
$L_{av}(v_1v_3,\varphi')\cap L_{av}(v_2v_4,\varphi')\neq \emptyset$, then color $v_1v_3$ and $v_2v_4$ with a same color from $L_{av}(v_1v_3,\varphi')\cap L_{av}(v_2v_4,\varphi')$.
If $L_{av}(v_1v_3,\varphi')\cap L_{av}(v_2v_4,\varphi')=\emptyset$, then we can extend $\varphi'$ to another partial coloring $\varphi''$ of $G$ by
coloring $v_1v_3$ and $v_2v_4$ with $\varphi''(v_1v_3)\in L_{av}(v_1v_3,\varphi')$ and $\varphi''(v_2v_4)\in L_{av}(v_2v_4,\varphi')$ so that $|L_{av}(v_1v_2,\varphi')\setminus \{\varphi''(v_1v_3),\varphi''(v_2v_4)\}|\geq 2$ and
$|L_{av}(v_3v_4,\varphi')\setminus \{\varphi''(v_1v_3),\varphi''(v_2v_4)\}|\geq 2$.
In either case, we can extend $\varphi'$ to another partial coloring $\varphi''$ of $G$ by
coloring $v_1v_3$ and $v_2v_4$ from their available lists, and moreover, $\varphi''$ satisfy
$$|L_{av}(u_1v_1,\varphi'')|,|L_{av}(u_3v_4,\varphi'')|,|L_{av}(v_1v_2,\varphi'')|,|L_{av}(v_3v_4,\varphi'')|\geq 2,$$
$$|L_{av}(v_2,\varphi'')|,|L_{av}(v_3,\varphi')|\geq 3,$$
$$|L_{av}(v_2v_3,\varphi'')|\geq 4,$$and
$$|L_{av}(u_1v_2,\varphi'')|,|L_{av}(u_2v_2,\varphi'')|,|L_{av}(u_2v_3,\varphi'')|,|L_{av}(u_3v_3,\varphi'')|\geq 5.$$
Therefore, by Lemma \ref{b}, $\varphi''$ can be extended to an $L$-total coloring $\varphi$ of $G$ without altering the assigned colors.

If $G$ contains $(c)$, then $G'=G-\{u_2,u_4\}$ has an $L'$-total coloring $\varphi'$ so that $L'(x)=L(x)$ for each $x\in VE(G')$.
Since there are at least two available colors for each edge of the cycle $u_1u_2u_3u_4$
and every 4-cycle is 2-edge-choosable, we can color $u_1u_2,u_2u_3,u_3u_4$ and $u_4u_1$ from their available lists so that the extended coloring is proper. At last, we color $u_2$ and $u_4$ from their available lists to obtain an $L$-total coloring $\varphi$ of $G$. This can be easily done since $u_2$ or $u_4$ is incident with four colored elements right row.

If $G$ contains $(d)$, then $G'=G-\{u_2,u_4\}$ has an $L'$-total coloring $\varphi'$ so that $L'(x)=L(x)$ for each $x\in VE(G')$.
One can check that $|L_{av}(u_2u_4,\varphi')|\geq 6$, $|L_{av}(u_1u_2,\varphi')|,|L_{av}(u_1u_4,\varphi')|\geq 2$, $|L_{av}(u_2u_3,\varphi')|,|L_{av}(u_3u_4,\varphi')|\geq 3$ and $|L_{av}(u_2,\varphi')|,|L_{av}(u_4,\varphi')|\geq 4$. By Lemma \ref{de}, $\varphi'$ can be extended to an $L$-total coloring $\varphi$ of $G$ without altering the colors in $G'$.

If $G$ contains $(e)$, then $G'=G-\{u_2,u_4,u_3v\}$ has an $L'$-total coloring $\varphi'$ so that $L'(x)=L(x)$ for each $x\in VE(G')$. We now erase the color on $v$ and still denote the current coloring by $\varphi'$. It is easy to check that
$|L_{av}(u_2u_4,\varphi')|\geq 6$, $|L_{av}(u_1u_2,\varphi')|,|L_{av}(u_1u_4,\varphi')|\geq 2$, $|L_{av}(u_2u_3,\varphi')|,|L_{av}(u_3u_4,\varphi')|\geq 3$, $|L_{av}(u_2,\varphi')|,|L_{av}(u_4,\varphi')|\geq 4$,
$|L_{av}(u_3v,\varphi')|\geq 2$ and $|L_{av}(v,\varphi')|\geq 3$. We now color $u_3v$ with $\varphi''(u_3v)\in L_{av}(u_3v,\varphi')$ and $v$ with $\varphi''(v)\in L_{av}(v,\varphi')\setminus \{\varphi''(u_3v)\}$
so that the extended partial coloring $\varphi''$ satisfies

(1)
$|L_{av}(u_2u_4,\varphi'')|\geq 6$, $|L_{av}(u_2,\varphi'')|, |L_{av}(u_4,\varphi'')|\geq 4$, $|L_{av}(u_1u_2,\varphi'')|,|L_{av}(u_2u_3,\varphi'')|,|L_{av}(u_3u_4,\varphi'')|,$\\$|L_{av}(u_1u_4,\varphi'')|\geq 2$

(2)
$L_{av}(u_1u_2,\varphi'')\neq L_{av}(u_2u_3,\varphi'')$ if $|L_{av}(u_1u_2,\varphi'')|=|L_{av}(u_2u_3,\varphi'')|=2$.

\noindent This can be done since $|L_{av}(u_3v,\varphi')|\geq 2$. Therefore, $\varphi'$ can be extended to an $L$-total coloring $\varphi$ of $G$ without altering the colors in $G'$ by Lemma \ref{de}.

If $G$ contains $(f)$, then $G'=G-\{u_2,u_4\}$ has an $L'$-total coloring $\varphi'$ so that $L'(x)=L(x)$ for each $x\in VE(G')$. We now erase the colors on $u_3,u_3v$ and $v$ and denote the current coloring by $\varphi''$. It is easy to see that
$$|L_{av}(u_2u_4,\varphi'')|\geq 6,$$
$$|L_{av}(u_2,\varphi'')|,|L_{av}(u_4,\varphi'')|\geq 5,$$
$$|L_{av}(u_2u_3,\varphi'')|,|L_{av}(u_3u_4,\varphi'')|\geq 4,$$ and $$|L_{av}(u_1u_2,\varphi'')|,|L_{av}(u_1u_4,\varphi'')|,|L_{av}(u_3v,\varphi'')|,|L_{av}(u_3,\varphi'')|,|L_{av}(v,\varphi'')|\geq 2$$
Since $\varphi''$ is a partial coloring of $\varphi'$, at~least~two~of~$L_{av}(u_3v,\varphi''),L_{av}(u_3,\varphi''),L_{av}(v,\varphi'')$~are~distinct~if
$|L_{av}(u_3v,\varphi'')|=|L_{av}(u_3,\varphi'')|=|L_{av}(v,\varphi'')|=2$. Therefore, by Lemma \ref{f}, $\varphi''$ can be extended to an $L$-total coloring $\varphi$ of $G$ without altering the colors in $G'$.

If $G$ contains $(g)$, then $G'=G-\{u_2,u_4\}$ has an $L'$-total coloring $\varphi'$ so that $L'(x)=L(x)$ for each $x\in VE(G')$. We now erase the colors on $u_1,u_3$ and $u_1u_3$ and denote the current coloring by $\varphi''$. One can see that
$$|L_{av}(u_2,\varphi'')|,|L_{av}(u_4,\varphi'')|,|L_{av}(u_2u_4,\varphi'')|\geq 6, $$ $$|L_{av}(u_1u_2,\varphi'')|,|L_{av}(u_2u_3,\varphi'')|,|L_{av}(u_3u_4,\varphi'')|,|L_{av}(u_1u_4,\varphi'')|\geq 4,$$ and $$|L_{av}(u_1,\varphi'')|,|L_{av}(u_3,\varphi'')|,|L_{av}(u_1u_3,\varphi'')|\geq 2.$$ Since $\varphi''$ is a partial coloring of $\varphi'$, at~least~two~of~$L_{av}(u_1u_3,\varphi''),L_{av}(u_1,\varphi''),L_{av}(u_3,\varphi'')$~are~distinct~if
$|L_{av}(u_1,\varphi'')|=|L_{av}(u_3,\varphi'')|=|L_{av}(u_1u_3,\varphi'')|=2$. Therefore, $\varphi''$ can be extended to an $L$-total coloring $\varphi$ of $G$ without altering the colors in $G'$ by Lemma \ref{g}.
\end{proof}

\end{document}